\documentclass[12pt]{amsart}

\usepackage{amssymb}
\usepackage{amsmath}
\usepackage{amsfonts}
\usepackage{graphicx}
\usepackage{amsthm}
\usepackage{enumerate}
\usepackage[mathscr]{eucal}
\usepackage{mathrsfs}
\usepackage{verbatim}
\usepackage{booktabs}
\usepackage{amssymb}

\makeatletter
\@namedef{subjclassname@2010}{%
  \textup{2010} Mathematics Subject Classification}
\makeatother

\numberwithin{equation}{section}

\theoremstyle{plain}
\newtheorem{theorem}{Theorem}[section]
\newtheorem{lemma}[theorem]{Lemma}

\theoremstyle{plain}

\numberwithin{equation}{section}

\theoremstyle{remark}
\newtheorem{remark}[theorem]{Remark}

\DeclareMathOperator{\dist}{dist}
\DeclareMathOperator{\vol}{vol}

\begin{document}

\date{}
\title
[Lattice Points]{Lattice points in model domains of finite type
in $\mathbb{R}^d$, II}
\author[J. Guo and T. Jiang]{Jingwei Guo \  \ Tao Jiang}

\address{Jingwei Guo\\Department of Mathematics\\
University of Science and Technology of China\\
Hefei, Anhui Province 230026, People's Republic of China}

\email{jwguo@ustc.edu.cn}

\address{Tao Jiang\\Department of Mathematics\\
University of Science and Technology of China\\
Hefei, Anhui Province 230026, People's Republic of China}

\email{jt1023@mail.ustc.edu.cn}

\thanks{J.G. is partially supported by the NSFC Grant No. 11571131 and No. 11501535.}

\begin{abstract}
We study the lattice point problem associated with a special class of
high-dimensional finite type domains via estimating the Fourier transforms
of corresponding indicator functions.
\end{abstract}



\maketitle

\section{Introduction}\label{introduction}

In this paper we study the lattice point problem associated with the following domain in  $\mathbb{R}^d$ ($d\geq 3$)
\begin{equation}
\mathcal{D}=\left\{x\in\mathbb{R}^{d}  :  \sum_{p=0}^{n-1}\left(\sum_{l=1+d_{p}}^{d_{p+1}} x_{l}^{\omega_{l}}
\right)^{m_{p+1}}\le 1\right\},\label{modeldomain}
\end{equation}
where $\omega_l\in 2\mathbb{N}$ with $1\leq l\leq d$, and $n$, $m_{p+1}, d_{p+1}\in \mathbb{N}$ with $0\leq p\leq n-1$ and $0=d_0< d_1< \ldots <d_{n-1}<d_n=d$.

Given any compact convex domain $\mathcal{B}\subset \mathbb{R}^d$ the associated lattice point problem is about counting the number of lattice points $\mathbb{Z}^d$
in the enlarged domain $t\mathcal{B}$ and the main problem is to study the remainder
$R_{\mathcal{B}}(t):=\#(t\mathcal{B}\cap\mathbb{Z}^d)-\vol(\mathcal{B})t^d$ for $t\geq 1$.

If the boundary $\partial \mathcal{B}$ has points of vanishing
Gaussian curvature, the problem is relatively not well understood.
The solution in high dimensions is still far from complete though a few partial results
are known.
For a better understanding we start with the study of some typical model domains of finite type (in the sense of Bruna, Nagel, and Wainger~\cite{BNW}) in $\mathbb{R}^d$ including
those appearing in \cite{model1} (see \eqref{model1} below) and more generally the domain $\mathcal{D}$ defined by \eqref{modeldomain}.

Our study of such domains is motivated by some examples in literature. To mention a few, super spheres
\begin{equation*}
\mathcal{B}=\{x\in \mathbb{R}^d :
|x_1|^{\omega}+|x_2|^{\omega}+\cdots+|x_d|^{\omega}\leq  1\}
\end{equation*}
are considered in Randol~\cite{randol} for even $\omega\geq 3$ and in Kr\"atzel~\cite{kratzel_odd} for odd $\omega\geq 3$, and it is proved that
\begin{equation}
R_{\mathcal{B}}(t)=O\left(t^{(d-1)(1-1/\omega)}+t^{d-2+2/(d+1)}\right)\label{conj-bound}
\end{equation}
and this estimate is the best possible when $\omega\geq d+1$. For further results of super spheres (ellipsoids) see \cite{kratzel, survey2004} and the references contained therein. Kr\"atzel~\cite{kratzel_2002} and Kr{\"a}tzel and
Nowak~\cite{K-N-2008,K-N-2011} study a special class of convex
domains in $\mathbb{R}^3$,
\begin{equation}
\mathcal{B}=\left\{x\in \mathbb{R}^3 :
|x_1|^{mk}+\left(|x_2|^k+|x_3|^k\right)^m\leq 1\right\}\label{KN-EX}
\end{equation}
with certain assumptions on reals $k$ and $m$ (for example, in \cite{K-N-2011}, $k>2$,
$m>1$, and $mk\geq 7/3$). The contribution of flat points is evaluated precisely and that of
other boundary points is estimated.

Motivated by these works the first author studied in \cite{model1} the domain
\begin{equation}
\mathcal{B}=\{x\in \mathbb{R}^d :
x_1^{\omega_1}+\cdots+x_d^{\omega_d}\leq 1\} \label{model1}
\end{equation}
for $\omega_l\in 2\mathbb{N}$ with $1\leq l\leq d$. A precise upper bound of $R_{\mathcal{B}}(t)$ is given, which leads to the same bound \eqref{conj-bound}.

In this paper we make a small progress by studying more general domain $\mathcal{D}$ in $\mathbb{R}^d$. For any $1\leq i\leq d$ there exists a unique $0\leq p(i)\leq n-1$ such that $1+d_{p(i)}\leq i\leq d_{p(i)+1}$. For any $1\leq j, l\leq d$, denote
\begin{equation}
m_{j,l}=\bigg\{ \begin{array}{ll}
1 & \textrm{if $p(j)=p(l)$},\\
m_{p(l)+1} & \textrm{if $p(j)\ne p(l)$}.
\end{array}\label{def-mjl}
\end{equation}
We then have
\begin{theorem}\label{latticethm}
For the domain $\mathcal{D}$ defined by \eqref{modeldomain}, we have
\begin{align}
\begin{split}
&R_{\mathcal{D}}(t) =\sum_{j=1}^{d}O,\Omega \left(t^{d-1-\sum\limits_{1\leq l\leq d, l\neq j}\frac{1}{m_{j,l}\omega_l}}\right)\\
&+\sum_{j=1}^{d}\sum_{i=2}^{d}\sum_{S\in P_{i}(\mathbb{N}_{d}), S\ni j} O\left(t^{d-1-\frac{i-1}{d+1}-\frac{2d}{d+1}\sum\limits_{1\leq l\leq d, l\notin S}\frac{1}{m_{j,l}\omega_l}}\right),\label{theorem1}
\end{split}
\end{align}
where $\mathbb{N}_{d}=\{1,2,\ldots,d\}$ and $P_{i}(\mathbb{N}_{d})$ is the collection of all subsets of $\mathbb{N}_{d}$ having $i$ elements. If $\omega=\max_{1\le j,l\le d}\{m_{j,l}\omega_l\}$, then
 \begin{equation}
|R_{\mathcal{D}}(t)|\lesssim t^{(d-1)(1-1/\omega)}+t^{d-2+2/(d+1)}.\label{theorem11}
\end{equation}
\end{theorem}

\begin{remark}
By taking all $m_{p+1}$'s being $1$ we recover the result in \cite{model1}. If we take that $d=3$, $n=2$, $m_1=m_2=m$, $\omega_1=\omega_2=\omega_3=k$, and $d_1=1$, the domain $\mathcal{D}$ is in the special form of \eqref{KN-EX}. Since we only consider domains with smooth boundary we did not allow exponents $m$ and $k$ to be real numbers, to the contrary Kr{\"a}tzel and
Nowak~\cite{K-N-2008,K-N-2011} do allow such general exponents. We mainly use harmonic analysis tools, while Kr{\"a}tzel and Nowak apply a ``cut-into-slices''-method to
reduce a three-dimensional problem to a two-dimensional one and then work carefully on the latter problem.
\end{remark}

\begin{remark}
Here we have the same phenomenon as in \cite{model1} (see Remark 1 in \cite{model1}): in \eqref{theorem1}, the first sum is the
contribution of boundary points which lie on coordinate axes; the terms for $i=d$ is
$O(t^{d-2+2/(d+1)})$, due to boundary points that are not on any
coordinate plane; all other terms for $2\leq i \leq d-1$ come from boundary points lying on
coordinate planes but not on axes.
\end{remark}

\begin{remark}
Many authors made efforts to study general domains (instead of special examples) in $\mathbb{R}^3$ under different curvature assumptions. Partial results are obtained by Kr\"atzel, Popov, Peter, Nowak, etc. We refer interested readers to two excellent survey articles \cite{survey2004, nowak14survey} and the references given there. For domains in high dimensions, satisfactory answers still wait to be found.
\end{remark}

\begin{remark}
For convex domains of finite type in $\mathbb{R}^d$ Iosevich, Sawyer, and
Seeger~\cite[Theorem 1.3]{I-S-S-0} provides an estimate of the remainder. Their results work for high dimensions and the curvature assumption looks quite neat. Unfortunately, even for some model domains, their bound may not be sharp. For example Randol's bound \eqref{conj-bound} (namely, \eqref{theorem11}) for
super spheres is better when $\omega$ is not too large (say, of size
$<2d^2+O(d)$).
\end{remark}

{\it Notations:} We set $\mathbb{Z}_{*}^{d}=\mathbb{Z}^{d}\setminus
\{0\}$, and $\mathbb{R}^d_*=\mathbb{R}^d\setminus \{0\}$. The
Fourier transform of any function $f\in L^1(\mathbb{R}^d)$ is
$\widehat{f}(\xi)=\int f(x) \exp(-2\pi i x\cdot \xi) \, dx$. For
functions $f$ and $g$ with $g$ taking nonnegative real values,
$f\lesssim g$ means $|f|\leq Cg$ for some constant $C$. If $f$ is
nonnegative, $f\gtrsim g$ means $g\lesssim f$. The Landau notation
$f=O(g)$ is equivalent to $f\lesssim g$. The notation $f\asymp g$
means that $f\lesssim g$ and $g\lesssim f$. For lower bounds,
$f(t)=\Omega_{+}(g(t))$ means that $\limsup(f(t)/g(t))>0$ as
$t\rightarrow \infty$, $f(t)=\Omega_{-}(g(t))$ stands for
$-f(t)=\Omega_{+}(g(t))$, and $f(t)=\Omega(g(t))$ means that at
least one of previous two assertions is  true.

\section{The Fourier transform of the indicator function $\chi_{\mathcal{D}}$}

Let $\mathcal{D}$ be defined by \eqref{modeldomain}. If $x\in
\partial \mathcal{D}$ let $T_x$ be the affine tangent plane to
$\partial \mathcal{D}$ at $x$. Bruna, Nagel, and Wainger~\cite{BNW}
define a ``ball''
\begin{equation*}
\tilde{B}(x, \delta)=\{y\in \partial \mathcal{D}: \dist(y,
T_x)<\delta\}
\end{equation*}
to be a cap near $x$ cut off from $\partial \mathcal{D}$ by a plane
parallel to $T_x$ at distance $\delta$ from it. For nonzero $\xi\in \mathbb{R}^d$ let $x(\xi)$ be the unique point
on $\partial \mathcal{D}$ where the unit exterior normal is
$\xi/|\xi|$.

We first prove a generalization of \cite[Lemma 2.2]{model1} concerning the size of the surface measure of $\tilde{B}(x(\xi),|\xi |^{-1})$.

\begin{lemma}\label{lemma1}
Let $0<\varepsilon_0\leq 1$ be a constant and $1\leq j\leq d$ an integer. For any nonzero $\xi\in \mathbb{R}^d$ with $|\xi_{j}|/|\xi|\ge\varepsilon_0$, we have
\begin{equation*}
\sigma \left(\tilde{B}(x(\xi),|\xi |^{-1})\right)\lesssim \prod_{
\substack{l=1 \\ l\ne{j}}}^{d} \min\left\{|\xi|^{-\frac{1}{m_{j,l}\omega_l}},|\xi|^{-\frac{1}{2}}\left(|\xi_l|/|\xi|\right)^{-\frac{m_{j,l}\omega_l-2}{2(m_{j,l}\omega_l-1)}}\right\},%\label{theoreminequalty1}
\end{equation*}
where $m_{j,l}$ is defined by \eqref{def-mjl} and the implicit constant only depends on $\varepsilon_0$ and $\mathcal{D}$.
\end{lemma}

\begin{proof}
For an arbitrarily fixed nonzero $\xi\in \mathbb{R}^d$  with $|\xi_{j}|/|\xi|\ge\varepsilon_0$, denote $x(\xi)=(a_1, a_2, \ldots, a_{d}) \in \partial\mathcal{D}$. Due to the symmetry of $\partial \mathcal D$, we may assume that all $\xi_{l}$'s and $a_l$'s are nonnegative. We only treat the case $j=1$ while all other cases are similar.

Denote $\partial \mathcal{D}$ by the equation
\begin{equation}
F(x)=0 \label{bbb}
\end{equation}
with $F$ explicitly determined by \eqref{modeldomain}. Hence
\begin{equation}
\frac{\nabla F}{|\nabla F|}(x(\xi))=\frac{\xi}{|\xi|}, \label{proofeq1}
\end{equation}
where  $|\nabla F|\asymp 1$.

By definition the interested cap is the one near $x(\xi)$ cut off from $\partial \mathcal{D}$ by the plane
\begin{equation}
\sum_{l=1}^{d}\xi_l(x_l-a_l)+1=0.\label{prooftangent plane}
\end{equation}
After changing variables $X_l=x_l-a_l$, combining equations \eqref{bbb} and \eqref{prooftangent plane}, and eliminating $X_1$, we get
\begin{equation}
\begin{split}
&\left(\left(a_1-\xi_1^{-1}-\xi_1^{-1}\sum_{l=2}^{d}\xi_l X_l\right)^{\omega_1}+\sum_{l=2}^{d_1}(a_l +X_l)^{\omega_l}\right)^{m_1}\\
&\qquad +\sum_{p=1}^{n-1}\left(\sum_{l=1+d_{p}}^{d_{p+1}} (a_l +X_l)^{\omega_{l}}
\right)^{m_{p+1}}-1=0.
\end{split}\label{proofprojection}
\end{equation}

To estimate $\sigma\left(\tilde{B}\left(x(\xi), |\xi |^{-1}\right)\right)$ it suffices to show that if  $(X_2, \ldots, X_d)$ satisfies \eqref{proofprojection} then for each $2\leq l\leq d$
\begin{equation}
\max |X_l|\lesssim \min\left\{|\xi|^{-\frac{1}{m_{1,l}\omega_l}},|\xi|^{-\frac{1}{2}}(|\xi_l|/|\xi|)^{-\frac{m_{1,l}\omega_l-2}{2(m_{1,l}\omega_l-1)}}\right\}.\label{estimate1}
\end{equation}
To prove \eqref{estimate1} we discuss two cases: $2\le l\le d_1$ or $1+d_1\le l \le d$.

\emph{Case 1}: $2\le l\le d_{1}$. We may assume $l=2$ while other cases can be handled similarly.

\emph{Subcase 1.1}: $\xi_2/|\xi|=0$. Then \eqref{proofprojection} implies
\begin{equation}
|X_2|\lesssim |\xi|^{-1/\omega_2}.\label{aaa}
\end{equation}

Indeed, in this case $a_2=0$ by \eqref{proofeq1}. We apply Taylor's expansion of order two to $(a_1-\xi_1^{-1}-\xi_1^{-1}\sum_{l=2}^{d}\xi_l X_l)^{\omega_1}$ at $a_1$ and to $(a_l+X_l)^{\omega_l}$ at $a_l$ (for $3\leq l\leq d$) with nonnegative remainders (due to the evenness of $\omega_l$). For each $0\leq p\leq n-1$ we then apply Taylor's expansion of order two to the $m_{p+1}$th powers in \eqref{proofprojection} at $\sum_{l=1+d_{p}}^{d_{p+1}} a_{l}^{\omega_{l}}$. After using the condition $x(\xi)\in \partial\mathcal D$ to cancel the constant term, \eqref{proofeq1} to eliminate linear terms, and dropping nonnegative remainder terms, we get
\begin{equation}
X_2^{\omega_2}\leq \omega_1a_1^{\omega_1-1}\xi_1^{-1},\label{ccc}
\end{equation}
which implies  \eqref{aaa} since $\xi_1\asymp |\xi|$.

\emph{Subcase 1.2}: $\xi_2/|\xi|\ne0$. In this case $a_2\neq 0$. Besides all the expansions used in Subcase 1.1, we also need
\begin{equation*}
(a_2 +X_2)^{\omega_2}=a_2^{\omega_2}+\omega_2a_2^{\omega_2-1}X_2+a_2^{\omega_2-2}X_2^{2}\left(\omega_2(\omega_2-1)/2+\delta_1\right)+X_2^{\omega_2} \end{equation*}
by the binomial formula, where
\begin{equation*}
\delta_1=C_{\omega_2}^{3}X_2/a_2+C_{\omega_2}^{4}(X_2/a_2)^{2}+\ldots+C_{\omega_2}^{\omega_2-1}(X_2/a_2)^{\omega_2-3}.
\end{equation*}
Like what we did in Subcase 1.1, we get
\begin{equation}
a_2^{\omega_2-2}X_2^{2}\left(\omega_2(\omega_2-1)/2+\delta_1\right)+X_2^{\omega_2}\le\omega_1a_1^{\omega_1-1}\xi_1^{-1}\label{proofineq2}
\end{equation}
as a replacement of \eqref{ccc}. Note that $a_1\gtrsim 1$ since $|\xi_{1}|/|\xi|\ge\varepsilon_0$. Hence the second equation in \eqref{proofeq1} implies $a_2\asymp(\xi_2/|\xi|)^{1/(\omega_2-1)}$.

If $X_2>0$, then $\delta_1>0$. \eqref{proofineq2} immediately implies the desired bound for $\max_{X_2>0}|X_2|$, namely
\begin{equation*}
\max_{X_2>0}|X_2|\lesssim \min\left\{|\xi|^{-\frac{1}{\omega_2}},|\xi|^{-\frac{1}{2}}(|\xi_2|/|\xi|)^{-\frac{\omega_2-2}{2(\omega_2-1)}}\right\}.
\end{equation*}

If $X_2<0$ and $\max_{X_2<0}|X_2|\leq c_1 a_2$ for a sufficiently small constant $c_1$ (say, such that $\omega_2(\omega_2-1)/2+\delta_1>\omega_2(\omega_2-1)/4$), then \eqref{proofineq2} implies the desired bound for $\max_{X_2<0}|X_2|$.

If $X_2<0$ and $\max_{X_2<0}|X_2|> c_1 a_2$, by a compactness argument there exists a constant $C_1$ (depending only on $c_1$ and $\mathcal{D}$) such that $\tilde{B}(x(\xi), C_1|\xi|^{-1})$ intersects the plane $x_2=-a_2$. It suffices to estimate the size of this larger cap $\tilde{B}(x(\xi), C_1|\xi|^{-1})$. Then we need to study \eqref{proofprojection} with $\xi$ replaced by $\xi/C_1$ and to estimate $\max_{X_2<0}|X_2|$ subject to $\max_{X_2<0}|X_2|> 2 a_2$. Like \eqref{proofineq2} we get
\begin{equation*}
a_2^{\omega_2-2}X_2^{2}\left(\omega_2(\omega_2-1)/2+\delta_1\right)+X_2^{\omega_2}\lesssim |\xi|^{-1}.
\end{equation*}
We also note that if $-X_2>2a_2$ then
\begin{align*}
&a_2^{\omega_2-2}X_2^{2}\left(\omega_2(\omega_2-1)/2+\delta_1\right)+X_2^{\omega_2}\\
&=(a_2 +X_2)^{\omega_2}-a_2^{\omega_2}-\omega_2a_2^{\omega_2-1}X_2\geq  (a_2 +X_2)^{\omega_2}\geq X_2^{\omega_2}/2^{\omega_2}.
\end{align*}
Combining these two inequalities above yields
\begin{equation}
\max_{X_2<0}|X_2|\lesssim |\xi|^{-1/\omega_2}.\label{ddd}
\end{equation}
Hence $a_2\lesssim |\xi|^{-1/\omega_2}$, which implies
\begin{equation}
|\xi|^{-\frac{1}{\omega_2}}\lesssim
|\xi|^{-\frac{1}{2}}\left(|\xi_2|/|\xi|\right)^{-\frac{\omega_2-2}{2(\omega_2-1)}}.\label{eee}
\end{equation}
By \eqref{ddd} and \eqref{eee} we get again the desired bound for $\max_{X_2<0}|X_2|$. This finishes Subcase 1.2, hence Case 1 as well.

\emph{Case 2}: $1+d_1\le l\le d$. We may assume $l=d$ while other cases can be handled similarly.

\emph{Subcase 2.1}: $\xi_{d}/|\xi|=0$. In this case $a_d=0$ by \eqref{proofeq1}. This case is the same as Subcase 1.1 except that we need to treat $X_d$ (instead of $X_2$) separately. More precisely we apply
\begin{equation*}
\left(\sum_{l=1+d_{n-1}}^{d} (a_l +X_l)^{\omega_{l}}
\right)^{m_{n}}\geq \left(\sum_{l=1+d_{n-1}}^{d-1}(a_l+X_l)^{\omega_l}\right)^{m_n}+X_d^{m_n\omega_d}
\end{equation*}
and then like \eqref{ccc} we get
\begin{equation*}
X_{d}^{m_n\omega_{d}}\leq m_1\omega_1\left(\sum_{l=1}^{d_1} a_l^{\omega_{l}}
\right)^{m_1-1}a_1^{\omega_1-1}\xi_1^{-1},
\end{equation*}
which implies
\begin{equation*}
\max|X_{d}|\lesssim|\xi|^{-1/(m_n\omega_{d})}.
\end{equation*}

\emph{Subcase 2.2}: $\xi_{d}/|\xi|\ne0$. Note that the last equation of \eqref{proofeq1} implies
\begin{equation}
\left(\sum_{l=1+d_{n-1}}^{d}a_{l}^{\omega_{l}}\right)^{m_n-1}a_{d}^{\omega_{d}-1}\asymp \xi_d/|\xi|,\label{ggg}
\end{equation}
hence
\begin{equation}
a_{d}\lesssim(\xi_{d}/{|\xi|})^{1/(m_n\omega_{d}-1)}.\label{fff}
\end{equation}

If $X_{d}>0$, we apply the binomial formula to $(a_d+X_d)^{\omega_d}$ and use
\begin{align*}
&\left(\sum_{l=1+d_{n-1}}^{d}(a_l+X_l)^{\omega_l}\right)^{m_n}\ge X_d^{m_n\omega_d}+\\
&\ \ \left(\sum_{l=1+d_{n-1}}^{d-1}(a_l+X_l)^{\omega_l}+a_d^{\omega_d}+\omega_da_d^{\omega_d-1}X_d+\frac{\omega_d(\omega_d-1)}{2}a_d^{\omega_d-2}X_d^2\right)^{m_n}
\end{align*}
to get a separated term $X_d^{m_n\omega_d}$. Like Subcase 1.1 we get
\begin{align*}
&\frac{m_n\omega_{d}(\omega_{d}-1)}{2}\left(\sum_{l=1+d_{n-1}}^{d}a_l^{\omega_l}\right)^{m_n-1}a_d^{\omega_d-2}X_d^2+X_d^{m_n\omega_d}\\
&\quad \leq m_1\omega_1\left(\sum_{l=1}^{d_1}a_l^{\omega_l}\right)^{m_1-1}a_1^{\omega_1-1}\xi_1^{-1}\lesssim |\xi|^{-1}.
\end{align*}
The inequality above, combining with \eqref{ggg} and \eqref{fff}, yields the desired bound for $\max_{X_d>0}|X_d|$, namely
\begin{equation*}
\max_{X_d>0}|X_d|\lesssim\min\left\{|\xi|^{-\frac{1}{m_n\omega_d}},|\xi|^{-\frac{1}{2}}(\xi_d/|\xi|)^{-\frac{m_n\omega_d-2}{2(m_n\omega_d-1)}}\right\}.
\end{equation*}

If $X_d<0$, we do not need to separate an $X_d^{m_n\omega_d}$ term. We mimic the computation to derive \eqref{proofineq2} in Subcase 1.2 and get
\begin{equation}
\begin{split}
&m_n\left(\sum_{l=1+d_{n-1}}^{d}a_{l}^{\omega_{l}}\right)^{m_n-1}\left(a_{d}^{\omega_{d}-2}X_{d}^{2}(\omega_{d}(\omega_{d}-1)/2+\delta_2)
+X_{d}^{\omega_{d}}\right)\\
&\quad \le m_1\omega_1\left(\sum_{l=1}^{d_1}a_l^{\omega_l}\right)^{m_1-1}a_1^{\omega_1-1}\xi_1^{-1},
\end{split}\label{proofineq5}
\end{equation}
where
\begin{equation*}
\delta_2=C_{\omega_d}^{3}X_d/a_d+C_{\omega_d}^{4}(X_d/a_d)^{2}+\ldots+C_{\omega_d}^{\omega_d-1}(X_d/a_d)^{\omega_3-3}.
\end{equation*}

If $\max_{X_{d}<0}|X_{d}|\le c_2a_{d}$ for a sufficiently small constant $c_2$, then \eqref{proofineq5} (with \eqref{ggg} and \eqref{fff}) implies
\begin{equation}
\max_{X_{d}<0}|X_{d}|\lesssim |\xi|^{-\frac{1}{2}}({\xi_{d}}/{|\xi|})^{-\frac{m_n\omega_{d}-2}{2(m_n\omega_{d}-1)}}\label{proofineq7}
\end{equation}
and
\begin{equation}
a_{d}^{\omega_d (m_n-1)}X_{d}^{\omega_{d}}\lesssim |\xi|^{-1}. \label{proofineq7.1}
\end{equation}

Since $\max_{X_{d}<0}|X_{d}|\le c_2a_{d}$, \eqref{proofineq7.1} implies
\begin{equation}
\max_{X_{d}<0}|X_{d}|\lesssim|\xi|^{-1/(m_n\omega_d)}. \label{proofineq7.3}
\end{equation}
The \eqref{proofineq7} and \eqref{proofineq7.3} give the desired bound for $\max_{X_{d}<0}|X_{d}|$ when $\max_{X_{d}<0}|X_{d}|\le c_2a_{d}$.

If $\max_{X_{d}<0}|X_{d}|> c_2a_{d}$, by a compactness argument there is a constant $C_2\ge1$ (depending only on $c_2$ and $\mathcal{D}$) such that $\tilde{B}(x(\xi),C_2|\xi |^{-1})$ intersects the plane $x_{d}=-a_{d}$. It suffices to estimate the size of the cap $\tilde{B}(x(\xi), C_2|\xi|^{-1})$. Then we need to study \eqref{proofprojection} with $\xi$ replaced by $\xi/C_2$ and to estimate $\max_{X_d<0}|X_d|$ subject to $\max_{X_d<0}|X_d|> 2 a_d$. Like \eqref{proofineq5}, we get
\begin{equation*}
m_n\left(\sum_{l=1+d_{n-1}}^{d}a_{l}^{\omega_{l}}\right)^{m_n-1}\left(a_{d}^{\omega_{d}-2}X_{d}^{2}(\omega_{d}(\omega_{d}-1)/2+\delta_2)
+X_{d}^{\omega_{d}}\right) \lesssim |\xi|^{-1}.
\end{equation*}
We also note that if $-X_d>2a_d$ then
\begin{equation*}
a_{d}^{\omega_{d}-2}X_{d}^{2}(\omega_{d}(\omega_{d}-1)/2+\delta_2)+X_{d}^{\omega_{d}}\geq X_d^{\omega_d}/2^{\omega_d}.
\end{equation*}
Combining these two inequalities above yields
\begin{equation}
\left(\sum_{l=1+d_{n-1}}^{d}a_{l}^{\omega_{l}}\right)^{m_n-1}\left(\max_{X_d<0}|X_d|\right)^{\omega_{d}}\lesssim |\xi|^{-1}.\label{proofineq9}
\end{equation}
It then follows from \eqref{proofineq9} and  $\max_{X_{d}<0}|X_{d}|> c_2a_{d}$ that
\begin{equation*}
\left(\sum_{l=1+d_{n-1}}^{d}a_{l}^{\omega_{l}}\right)^{m_n-1}a_{d}^{\omega_{d}-2}\left(\max_{X_d<0}|X_d|\right)^{2}\lesssim |\xi|^{-1},
\end{equation*}
which (with \eqref{ggg} and \eqref{fff})  implies \eqref{proofineq7}.

It remains to prove \eqref{proofineq7.3}. Since the cap $\tilde{B}(x(\xi),C_2|\xi |^{-1})$ intersects the coordinate plane $x_{d}=0$, we can take a point $P$ from the intersection. By \cite[Theorem A]{BNW} there exists a constant $C_3$ (depending only on $\mathcal{D}$) such that $\tilde{B}(x(\xi),C_2|\xi |^{-1})\subset \tilde{B}(P, C_3C_2|\xi |^{-1})$. Applying to $\tilde{B}(P, C_3C_2|\xi |^{-1})$ the result of Subcase 2.1 yields \eqref{proofineq7.3}. This finishes the estimate of $\max_{X_{d}<0}|X_{d}|$ when $\max_{X_{d}<0}|X_{d}|>c_2a_{d}$ and the proof of Subcase 2.2, hence the entire proof of the lemma.
\end{proof}

It follows
easily from the Gauss-Green formula, \cite[Theorem B]{BNW}, and Lemma \ref{lemma1} to get the following generalization of \cite[II, Theorem 2]{randol} and \cite[Theorem 2.1]{model1}.

\begin{theorem}\label{theorem2}
Let $0<\varepsilon_0\leq 1$ be a constant and $1\leq j\leq d$ an integer. For any $\xi\in S^{d-1}$ with $|\xi_{j}|\ge\varepsilon_0$ and $t>0$ we have
\begin{equation*}
|\widehat{\chi}_{\mathcal D}(t\xi)|\lesssim t^{-1}\prod_{
\substack{l=1 \\ l\ne{j}}}^{d} \min\left\{t^{-\frac{1}{m_{j,l}\omega_l}},t^{-\frac{1}{2}}|\xi_l|^{-\frac{m_{j,l}\omega_l-2}{2(m_{j,l}\omega_l-1)}}\right\},
\end{equation*}
where $m_{j,l}$ is defined by \eqref{def-mjl} and the implicit constant only depends on $\varepsilon_0$ and $\mathcal{D}$.
\end{theorem}

\section{Proof of Theorem \ref{latticethm}}

 \begin{proof}
We start with a standard inequality
\begin{equation}
\chi_{(t-\epsilon)\mathcal{D}}*\rho_{\epsilon}\le\chi_{t\mathcal{D}}\le\chi_{(t+\epsilon)\mathcal{D}}*\rho_{\epsilon},\label{hhh}
\end{equation}
where $0\le\rho\in C_{0}^{\infty}(\mathbb{R}^{d})$ satisfies $\int_{\mathbb{R}^{d}}\rho{(x)}\,\mathrm{d}x=1$ and $\rho_{\epsilon}(x)=\epsilon^{-d}\rho(\epsilon^{-1}x)$ with $\epsilon>0$. By summing \eqref{hhh} over $\mathbb{Z}^{d}$ and using the Poisson summation formula we get
\begin{equation}
\sum_{k\in\mathbb{Z}^{d}}\widehat{\chi}_{(t-\epsilon)\mathcal{D}}(k)\widehat{\rho}(\epsilon k)\le\sum_{k\in\mathbb{Z}^{d}}\chi_{t\mathcal{D}}(k)\le\sum_{k\in\mathbb{Z}^{d}}\widehat{\chi}_{(t+\epsilon)\mathcal{D}}(k)\widehat{\rho}(\epsilon k).\label{applyintrodution2}
\end{equation}
Note that
\begin{equation}
\begin{split}
&\sum_{k\in\mathbb{Z}^{d}}\widehat{\chi}_{(t\pm\epsilon)\mathcal{D}}(k)\widehat{\rho}(\epsilon k)\\
&\quad =\vol(\mathcal{D})t^{d}+O(t^{d-1}\epsilon)+(t\pm\epsilon)^{d}\sum_{k\in\mathbb{Z}_{*}^{d}}\widehat{\chi}_\mathcal{D}((t\pm\epsilon)k)\widehat{\rho}(\epsilon k).
\end{split}\label{applyintrodution3}
\end{equation}
Hence we need to estimate $\sum_{k\in\mathbb{Z}_{*}^{d}}\widehat{\chi}_\mathcal{D}(tk)\widehat{\rho}(\epsilon k)$. By using a partition of unity we have
\begin{equation*}
\sum_{k\in\mathbb{Z}_{*}^{d}}\widehat{\chi}_\mathcal{D}(tk)\widehat{\rho}(\epsilon k)=\sum_{j=1}^{d}\sum_{k\in\mathbb{Z}_{*}^{d}}\Omega_j(k)\widehat{\chi}_\mathcal{D}(tk)\widehat{\rho}(\epsilon k)=:\sum_{j=1}^{d}S_j,
\end{equation*}
where $\Omega_j$ is supported in $\Gamma_j=\{x\in\mathbb{R}^{d} : |x_j|/|x|\ge(2d)^{-1/2}\}$ and smooth away from the origin. We then split $S_j$ as follows
\begin{equation*}
S_j=\sum_{i=1}^{d}\sum_{(i)}\Omega_j(k)\widehat{\chi}_\mathcal{D}(tk)\widehat{\rho}(\epsilon k)=:\sum_{i=1}^{d}S_{i,j},
\end{equation*}
where $\sum_{(i)}$ means the summation is over all $k\in\mathbb{Z}_{*}^{d}$ having exactly $i$ nonzero components.

Now we estimate $S_1$. The definition of $\Omega_1$ restricts the domain of summation to a cone about $x_1$-axis such that $|k_1|/|k|\ge(2d)^{-1/2}$. Applying Theorem \ref{theorem2} (with $\varepsilon_0=(2d)^{-1/2}$) yields
\begin{equation}
|S_{1,1}|\lesssim\sum_{k_1\in \mathbb{Z}_{*}^{1}}|tk_1|^{-1}\prod_{l=2}^{d}|tk_1|^{-1/(m_{1,l}\omega_l)}\lesssim t^{-1-\sum_{l=2}^{d}1/(m_{1,l}\omega_l)}.\label{applyineq1}
\end{equation}
For $2\le i\le d$, by applying Theorem \ref{theorem2} and comparing the sums with integrals in polar coordinates we have
\begin{equation}
|S_{i,1}|\lesssim \sum_{S\in P_{i}(\mathbb{N}_{d}),S\ni 1}t^{-\frac{i+1}{2}-\sum_{l=1,l\notin S}^{d}\frac{1}{m_{1,l}\omega_l}}\left(1+\epsilon^{-\frac{i-1}{2}+\sum_{l=1,l\notin S}^{d}\frac{1}{m_{1,l}\omega_l}}\right).\label{applyineq2}
\end{equation}
Note that the first term of the right side above is less than the bound of $|S_{1,1}|$ in \eqref{applyineq1}. We take $\epsilon=t^{-(d-1)/(d+1)}$. Then \eqref{applyineq1} and \eqref{applyineq2} give
\begin{equation*}
|S_1|\lesssim t^{-1-\sum_{l=2}^{d}\frac{1}{m_{1,l}\omega_l}}+\sum_{i=2}^{d}\sum_{S\in P_{i}(\mathbb{N}_{d}),S\ni 1}t^{-1-\frac{i-1}{d+1}-\frac{2d}{d+1}\sum_{l=1,l\notin S}^{d}\frac{1}{m_{1,l}\omega_l}}.
\end{equation*}

The estimations of $S_j$ for $2\leq j\leq d$ are similar. Then we obtain a bound of $\sum_{k\in\mathbb{Z}_{*}^{d}}\widehat{\chi}_\mathcal{D}(tk)\widehat{\rho}(\epsilon k)$. Thus combining \eqref{applyintrodution2}, \eqref{applyintrodution3} and the bound of $\sum_{k\in\mathbb{Z}_{*}^{d}}\widehat{\chi}_\mathcal{D}(tk)\widehat{\rho}(\epsilon k)$ yields the desired upper bound in \eqref{theorem1}, from which we can derive \eqref{theorem11} easily.

It remains to prove the lower bound in \eqref{theorem1} (see also
\cite[P.167-168]{I-S-S-0}). We may assume $j=1$ while other cases are similar.

We first apply the asymptotic expansion in Schulz~\cite{schulz} to get
 \begin{equation*}
 \widehat{n_1d\sigma}(tk)=C_4i\sin(-2\pi tk_1+\pi \nu/2)(tk_1)^{-\nu}+O\left((tk_1)^{-\nu-1/\eta}\right),
 \end{equation*}
 where $n_1$ is the first component of the Gauss map of $\partial \mathcal{D}$, $d\sigma$ is the induced Lebesgue measure on $\partial \mathcal{D}$, $k=(k_1,0,\ldots,0)$ with $k_1\in\mathbb{N}$, $\nu=\sum_{l=2}^{d}1/(m_{1,l}\omega_l)$, $C_4$ is a real number (depending on $\mathcal D$ and the fixed direction of $k$), and $\eta$ is the least common multiple of $m_{1,2}\omega_2,\ldots,m_{1,d}\omega_d$.
Hence by the Gauss-Green formula we have
  \begin{equation}
 \widehat{\chi}_{\mathcal{D}}(tk)=C_5\sin(-2\pi tk_1+\pi \nu/2)(tk_1)^{-1-\nu}+O\left((tk_1)^{-1-\nu-1/\eta}\right).\label{fourierexpansion}
 \end{equation}

Then we split $S_{1,1}$ as follows
\begin{equation*}
S_{1,1}=\sum_{\substack{k=(k_1,0,\ldots,0)\\
k_1\in \mathbb{Z}_{*}^{1}}}\widehat{\chi}_\mathcal{D}(tk)
+\sum_{\substack{k=(k_1,0,\ldots,0)\\k_1\in \mathbb{Z}_{*}^{1}}}\widehat{\chi}_\mathcal{D}(tk)(\widehat{\rho}(\epsilon k)-1).
 \end{equation*}
From \eqref{fourierexpansion}, we have
 \begin{equation}
\sum_{\substack{k=(k_1,0,\ldots,0)\\
k_1\in \mathbb{Z}_{*}^{1}}}\widehat{\chi}_\mathcal{D}(tk)=t^{-1-\nu}g(t)+O(t^{-1-\nu-1/\eta})\label{lowerbound1},
\end{equation}
where
\begin{equation*}
g(t)=\sum_{k_1\in\mathbb{Z}_{*}^{1}}C_5\sin(-2\pi t|k_1|+\pi \nu/2)|k_1|^{-1-\nu}.
\end{equation*}
Here the real function $g(t)$ is periodic and not identically zero, so we have $\limsup_{t\to\infty}|g(t)|>0$. And
\begin{equation}
\sum_{\substack{k=(k_1,0,\ldots,0)\\k_1\in \mathbb{Z}_{*}^{1}}}\widehat{\chi}_\mathcal{D}(tk)(\widehat{\rho}(\epsilon k)-1)=O\left(t^{-1-\nu}(\epsilon^{\nu}+\epsilon)\right).\label{lowerbound2}
\end{equation}
Combining \eqref{lowerbound1}, \eqref{lowerbound2}, \eqref{applyintrodution2} and \eqref{applyintrodution3} yields the desired lower bound. This finishes the proof.

\end{proof}


\begin{thebibliography}{99}




\bibitem{BNW}
J. Bruna, A. Nagel and S. Wainger, Convex hypersurfaces and Fourier
transforms, \textit{Ann. of Math. (2)} \textbf{127}, 333--365, 1988.

\bibitem{model1}
J. Guo, A note on lattice points in model domains of finite type in $\mathbb{R}^d$,
\textit{Arch. Math. (Basel)} \textbf{108}, 45--53, 2017.


\bibitem{I-S-S-0}
A. Iosevich, E. Sawyer, and A. Seeger, Two problems associated with
convex finite type domains, \textit{Publ. Mat.} \textbf{46}, no. 1,
153--177, 2002.

\bibitem{survey2004}
A. Ivi\'c, E. Kr\"atzel, M. K\"uhleitner, and W. G. Nowak, Lattice
points in large regions and related arithmetic functions: recent
developments in a very classic topic, Elementare und analytische
Zahlentheorie, 89--128, Schr. Wiss. Ges. Johann Wolfgang Goethe
Univ. Frankfurt am Main, 20, Franz Steiner Verlag Stuttgart,
Stuttgart, 2006.



\bibitem{kratzel_odd}
E. Kr{\"a}tzel, Mittlere Darstellungen nat\"urlicher Zahlen als
Summe von $n$ $k$-ten Potenzen. (German), \textit{Czechoslovak
Math. J.} \textbf{23(98)}, 57--73, 1973.


\bibitem{kratzel}
E. Kr\"{a}tzel, Lattice Points, Kluwer Academic Publishers Group, Dordrecht, 1988.


\bibitem{kratzel_2002}
E. Kr{\"a}tzel, Lattice points in some special three-dimensional
convex bodies with points of Gaussian curvature zero at the
boundary, \textit{Comment. Math. Univ. Carolinae} \textbf{43}, no.
4, 755--771, 2002.


\bibitem{K-N-2008}
E. Kr{\"a}tzel and W. G. Nowak, The lattice discrepancy of bodies
bounded by a rotating Lam\'{e}'s curve, \textit{Monatsh. Math.}
\textbf{154}, no. 2, 145--156, 2008.

\bibitem{K-N-2011}
E. Kr{\"a}tzel and W. G. Nowak, The lattice discrepancy of certain
three-dimensional bodies, \textit{Monatsh. Math.} \textbf{163}, no.
2, 149--174, 2011.


\bibitem{nowak14survey}
W. G. Nowak, Integer points in large bodies, Topics in mathematical analysis and applications, 583–-599,
Springer Optim. Appl., 94, Springer, Cham, 2014.


\bibitem{randol}
B. Randol, A lattice-point problem, \textit{Trans.\ Amer.\ Math.\
Soc.} \textbf{121}, 257--268, 1966; A lattice-point problem. II,
\textit{Trans.\ Amer.\ Math.\ Soc.} \textbf{125}, 101--113, 1966.


\bibitem{schulz}
H. Schulz, Convex hypersurfaces of finite type and the asymptotics
of their Fourier transforms, \textit{Indiana Univ.\ Math.\ J.}
\textbf{40}, no. 4, 1267--1275, 1991.


\end{thebibliography}
\end{document}